\documentclass[a4paper,12pt]{amsart}
\usepackage{pst-node}
\usepackage{amsmath,amsthm,amssymb,a4wide}
\usepackage[arrow,matrix]{xy}

\usepackage{ulem}
\usepackage{graphicx}
\newgray{lightgray}{0.9}

\theoremstyle{ams}
\newtheorem{theorem}{Theorem}[section]
\newtheorem{proposition}[theorem]{Proposition}

\theoremstyle{definition}

\newtheorem{conjecture}[theorem]{Conjecture}

\newtheorem{remark}[theorem]{Remark}

\numberwithin{figure}{section} \numberwithin{equation}{section}

\newcommand{\Q}{\mathbb{Q}}
\newcommand{\R}{\mathbb{R}}

\begin{document}
\title[Unimodality of the Betti numbers]{Unimodality of the Betti numbers for Hamiltonian circle action with isolated fixed points}

\author[Y. Cho]{Yunhyung Cho}
\address{Department of Mathematical Sciences, KAIST, 335 Gwahangno,
Yu-sung Gu, Daejeon 305-701, Korea}
\email{yh\_cho@kaist.ac.kr}
\author[M. K. Kim]{Min Kyu Kim*}
\address{Department of Mathematics Education,
Gyeongin National University of Education, San 59-12, Gyesan-dong,
Gyeyang-gu, Incheon, 407-753, Korea}
\email{mkkim@kias.re.kr}

\thanks{* supported by GINUE research fund}

\maketitle

\begin{abstract}

Let $(M,\omega)$ be an eight-dimensional closed
symplectic manifold equipped with a Hamiltonian
circle action with only isolated fixed points. In
this article, we will show that the Betti numbers
of $M$ are unimodal, i.e. $b_0(M) \leq b_2(M)
\leq b_4(M)$.

\end{abstract}

\section{introduction}

Let $(M,\omega,J)$ be an $n$-dimensional closed
K\"{a}hler manifold. Then $(M,\omega,J)$
satisfies the hard Lefschetz property so that the
Betti numbers are unimodal, i.e.
$$b_i(M) \leq b_{i+2}(M)$$ for all $i < n$.
In symplectic case, the unimodality of the Betti
numbers is obviously not clear in general. In
this paper, we will consider the following
conjectural question which is addressed in
\cite{JHKLM}.

\begin{conjecture} \label{conjecture1.1}
Let $(M,\omega)$ be a closed symplectic manifold
with a Hamiltonian circle action. Assume that all
fixed points are isolated. Then the Betti numbers
are unimodal.
\end{conjecture}

The reason why we put the condition
``\textit{isolated fixed points}" is that, as far
as the authors know, all known examples of
Hamiltonian circle action with only isolated
fixed points admit a K\"{a}hler structure. In
particular, Y. Karshon \cite{Ka} proved that
every symplectic 4-manifold with Hamiltonian
circle action with only isolated fixed points is
symplectomorphic to some smooth projective toric
variety.
    In this paper, we will show that
    \begin{theorem}\label{main}
        Let $(M,\omega)$ be an 8-dimensional closed symplectic manifold equipped with a Hamiltonian circle action with only isolated fixed points. Then the Betti numbers of $M$ are unimodal, i.e.
        $b_0(M) \leq b_2(M) \leq b_4(M)$.
    \end{theorem}

We would like to give a remark that our method to
approach this problem is purely ``topological" in
the sense that we do not use any geometric
structure, like an almost complex structure or
metric. Moreover, we cannot be sure whether our
method does work in higher dimensional cases.

\section{Equivariant chomology}

In this section, we briefly review an elementary
equivariant cohomology theory and the
localization theorem for a circle action which
will be used in Section 3. Let $S^1$ be a unit
circle group and let $M$ be an $S^1$-manifold.
Then the equivariant cohomology $H^*_{S^1}(M)$ is
defined by $$H^*_{S^1}(M) := H^*(M \times_{S^1}
ES^1)$$ where $ES^1$ is a contractible space on
which $S^1$ acts freely. Since $M \times_{S^1}
ES^1$ has a natural $M$-bundle structure over the
classifying space $BS^1 := ES^1 / S^1$, the
equivariant cohomology $H^*_{S^1}(M)$ admits a
$H^*(BS^1)$-module structure. Note that
$H^*(BS^1; \R)$ is isomorphic to the polynomial
ring $\R[u]$ where $u$ is of degree two. For the
fixed point set $M^{S^1}$, the inclusion map $i :
M^{S^1} \hookrightarrow M$ induces a
$H^*(BS^1)$-algebra homomorphism $$i^* :
H^*_{S^1}(M) \rightarrow H^*_{S^1}(M^{S^1}) \cong
\bigoplus_{F \subset M^{S^1}} H^*(F)\otimes
H^*(BS^1)$$ and we call $i^*$ a
\textit{restriction map to the fixed point set}.
Note that for an inclusion $i_F : F
\hookrightarrow M^{S^1}$, it induces a natural
projection $i_F^* : H^*_{S^1}(M^{S^1})
\rightarrow H^*_{S^1}(F) \cong H^*(F)\otimes
H^*(BS^1).$ For every $\alpha \in H^*_{S^1}(M)$,
we will denote by $\alpha|_F$ an image $i_F^* (
i^*(\alpha)).$ The main technique for proving
Theorem \ref{main} is the following, which is
called Atiyah-Bott-Berlin-Vergne localization
theorem.

\begin{theorem}[A-B-B-V localization theorem]\label{localization}
Let $M$ be a compact manifold with a circle
action with isolated fixed points. Let $\alpha
\in H^*_{S^1}(M)$. Then as an element of $\Q(u)$,
we have
    $$\int_M \alpha = \sum_{F \in M^{S^1}} \frac{\alpha|_F}{e_F}$$ where the sum is taken over all fixed points, and $e_F$ is the equivariant Euler class of the normal bundle to $F$.
\end{theorem}

\begin{remark}
Sometimes, the integral $\int_M$ is called an
\textit{integration along the fiber $M$}. In
fact, as a Cartan model, every equivariant
cohomology class can be written as a sum of the
form $x \otimes u^k \in H^*_{S^1}(M) \cong H^*(M)
\otimes H^*(BS^1)$ and the operation $\int_M$
acts on the ordinary cohomology factor. Hence if
$\alpha \in H^*_{S^1}(M)$ is of degree less than
a dimension of $M$, then we have $$\int_M \alpha
= \sum_{F \in M^{S^1}} \frac{\alpha|_F}{e_F} =
0.$$
\end{remark}

When our manifold has a symplectic structure $\omega$ and the action preserves $\omega$, then the equivariant cohomology satisfies a remarkable property as follows.

\begin{theorem}\cite{Ki}\label{injectivity}
Let $(M,\omega)$ be a closed symplectic manifold
and $S^1$ act on $(M,\omega)$ in a Hamiltonian
fashion. Then the restriction map $i^*$ to the
fixed point set is injective.
\end{theorem}

Theorem \ref{injectivity} enables us to study the
ring structure of $H^*_{S^1}(M)$ more easily via
the restriction map. For instance, assume that
all fixed points are isolated. Then
$H^*_{S^1}(M^{S^1})$ is nothing but $\bigoplus_{F
\in M^{S^1}} H^*(BS^1) \cong \bigoplus_{F \in
M^{S^1}} \R[u]$ where $u$ is a degree two
generator of $H^*(BS^1)$. Hence we can think of
an element $f \in H^*_{S^1}(M)$ as a function
$i^*(f)$ from the fixed point set $M^{S^1}$ to
the polynomial ring $\R[u]$ with one-variable
$u$. Also, for any elements $f$ and $g$ of
$H^*_{S^1}(M)$, the product $f \cdot g$ can be
computed by studying $i^*(f \cdot g)$, which is
simply the product of $i^*(f)$ and $i^*(g)$
component-wise.

Now, consider a Hamiltonian $S^1$-manifold
$(M,\omega)$ with a moment map $H : M \rightarrow
\R$. Then we may construct an equivariant
symplectic class on the Borel construction as
follows. For the product space $M \times ES^1$,
consider a two form $\omega_H := \omega + d(H
\cdot \theta)$, regarding $\omega$ as a pull-back
of $\omega$ along the projection $M \times ES^1
\rightarrow M$ and $\theta$ as a pull-back of a
connection 1-form on $ES^1$ along the projection
$M \times ES^1 \rightarrow ES^1$. It is not hard
to show that $\omega_H$ is $S^1$-invariant and
vanishes on the fiber $S^1$ over $M \times_{S^1}
ES^1$. So we may push-forward $\omega_H$ to the
Borel construction $M \times_{S^1} ES^1$ and
denote by $\widetilde{\omega}_H$ the push-forward
of $\omega_H$. Obviously, the restriction of
$\widetilde{\omega}_H$ on each fiber $M$ is
precisely $\omega$ and we call a class
$[\widetilde{\omega}_H] \in H^2_{S^1}(M)$ an
\textit{equivariant symplectic class with respect
to $H$.} By definition of $\widetilde{\omega}_H$,
we have the following proposition.

\begin{proposition}\label{easy}\cite{Au}
Let $F \in M^{S^1}$ be an isolated fixed point of
the given Hamiltonian circle action. Then we have
$$[\widetilde{\omega}_H]|_F = H(F)u.$$
\end{proposition}

\section{main theorem}

Let $(M,\omega)$ be a closed symplectic manifold,
$S^1$ be the unit circle group acting on
$(M,\omega)$ in a Hamiltonian fashion, and $H : M
\rightarrow \R$ be a moment map for the given
action. For each connected fixed component $F
\subset M^{S^1}$, let $k_F$ be the index of $F$
with respect to $H$. Let $\nu_{F}$ be a normal
bundle of $F$ in $M$. Then the \textit{negative
normal bundle} $\nu^{-}_{F}$ of $F$ is defined by
a sub-bundle of $\nu_F$ whose fiber is contained
in an unstable submanifold of $M$ at $F$ with
respect to $H$. We denote by $e^{-}_F \in
H^*_{S^1}(F)$ the equivariant Euler class of
$\nu^{-}_F$. D. McDuff and S. Tolman found a
remarkable family of equivariant cohomology
classes as follows.

\begin{theorem}\cite{McT}\label{McT}
Let $(M,\omega)$ be a closed symplectic manifold
equipped with a Hamiltonian circle action with a
moment map $H : M \rightarrow \R$. Let $F$ be a
fixed component of the action. Then given any
cohomology class $Y \in H^i(F)$, there exists the
unique class $\widetilde{Y} \in H^{i + k_F}(M)$
such that
\begin{enumerate}
\item the restriction of $\widetilde{Y}$ to
$M^{<H(F)}$ vanishes,
\item $\widetilde{Y}|_F = Y \cup e^{-}_F$, and
\item the degree of
$\widetilde{Y}_{F'} \in H^*_{S^1}(F')$
is less than the index $k_{F'}$ of $F'$
for all fixed components $F' \neq F.$
\end{enumerate}
\end{theorem}

We call such class $\widetilde{Y}$ a
\textit{canonical class} with respect to $Y$. In
this case when all fixed points are isolated, let
$F$ be a fixed point of index $k_F$ and $1_F \in
H^0(F)$ be the identity element of $H^*(F)$. Then
Theorem \ref{McT} implies that there exists the
unique class $\alpha_F \in H^{k_F}_{S^1}(M)$ such
that
\begin{enumerate}
\item $\alpha_F |_{F'} = 0$ for every
$F' \in M^{S^1}$ with $H(F') < H(F)$,
\item $\alpha_F |_F = e^{-}_F = \prod w^{-}_iu$,
where $w^{-}_i$ is the negative weight of
$S^1$-representation on $T_F M$
for $i=1, \cdots, \frac{k_F}{2}$,
\item $\alpha_F |_{F'} = 0$
for every $F' \ne F \in M^{S^1}$ with $k_{F'}
\leq k_F$.
\end{enumerate}

Now, we prove our main theorem.

\begin{proof}[Proof of Theorem \ref{main}]
By the connectivity of $M$, the first inequality
is obvious. Now, let's assume that $b_2(M) >
b_4(M)$. Let $z_1, \cdots, z_k$ be the fixed
points of index 2. Also, we denote by $\alpha_i
\in H^2_{S^1}(M)$ the canonical class with
respect to $z_i$ where $k=b_2(M)$. Then our
assumption $b_2(M) > b_4(M)$ implies that there
is a non-zero class $\alpha = \sum_{i=1}^{k} c_i
\alpha_i$ such that $\alpha |_ F = 0$ for every
fixed point $F$ of index 4. In other words,
$\alpha$ can survive only on the fixed points of
index 2, 6, or 8. Now, consider an equivariant
symplectic class $\widetilde{\omega}_H$ such that
the maximum of $H$ is zero. If we let $\beta =
\alpha^2 \cdot \widetilde{\omega}$, then for each
fixed point $F$, the restriction $\beta|_F =
(\alpha|_F)^2 \cdot H(F)u$ by Proposition
\ref{easy}. Since we took our moment map
satisfying $\max H = 0$, we have $H(F) < 0$ for
every fixed point $F$ which is not maximal. Hence
all coefficients of $\beta|_F$ is non-positive.
In particular, $\beta|_F$ is non-zero for some
fixed point $F$ of index 2. Hence  $\beta$ is
non-zero class of degree 6 in $H^*_{S^1}(M)$ and
survives only on the fixed points of index 2 or
6. Applying the localization theorem to $\beta$,
we have
$$ 0 = \int_M \beta
= \sum_{F \in M^{S^1}} \frac{\beta|_F}{e_F} =
\sum_{F \in M^{S^1}, \atop \textrm{ind}(F)=2}
\frac{\beta|_F}{e_F} + \sum_{F \in M^{S^1}, \atop
\textrm{ind}(F)=6}\frac{\beta|_F}{e_F} .$$ Since
the coefficient of $u^4$ of $e_F$ is negative for
every fixed point $F$ of index 2 or 6, it is a
contradiction.

\end{proof}

\end{document}